\documentclass[reqno]{amsart}
\usepackage{hyperref}
\usepackage{amsmath,amssymb,amsthm}
\usepackage{graphicx,tikz}
\newtheorem{theorem}{Theorem}

\newtheorem{corollary}{Corollary}

\newtheorem{lemma}{Lemma}

\theoremstyle{plain}

\newtheorem{assumption}{Assumption}

\theoremstyle{plain}

\begin{document}

\title[]{Hitting times in the stochastic block model}

\author[]{Andrea Ottolini}
\address[]{Department of Mathematics, University of Washington, Seattle, WA 98195, USA}
\email{ottolini@uw.edu}

\thanks{}

\begin{abstract} 
Given a large connected graph $G=(V,E)$, and two vertices $w,\neq v$, let $T_{w,v}$ be the first hitting time to $v$ starting from $w$ for the simple random walk on $G$. We prove a general theorem that guarantees, under some assumptions on $G$, to approximate $\mathbb E[T_{w,v}]$ up to $o(1)$ terms. As a corollary, we derive explicit formulas for the stochastic block model with two communities and connectivity parameters $p$ and $q$, and show that the average hitting times, for fixed $v$ and as $w$ varies, concentrates around four possible values. The proof is purely probabilistic and make use of a coupling argument.
\end{abstract}

\maketitle
\section{Introduction}
\subsection{The Erd\H{o}s-Rényi case} 
Let $G$ be a simple connected graph, equipped with simple random walk. Given two vertices $w\neq v$, let $T_{w,v}$ be the first time that simple random walk, starting from $w$, hits $v$. Our goal is to understand $\mathbb E[T_{w,v}]$, which we refer to as the average hitting time from $w$ to $v$. Our motivation comes from a recent joint work with Steinerberger \cite{ott}, where we show that if $G$ is sampled according to the Erd\H{o}s-Rényi model with parameters $n$ (large) and $p$ (fixed), then there is an almost exact formula: with high probability, for every pair $w\neq v\in V$
\begin{equation}\label{steiner}
\mathbb E[T_{w,v}]=\frac{2|E|}{\deg v}-1+\frac{1}{p}1_{w\not\sim v}+\mathcal O\left(\sqrt\frac{\log^3 n}{n}\,\right).
\end{equation}
Here, $E$ denotes the number of edges, $\deg(v)$ the number of neighbors of $v$, and $1_{w\not\sim v}$ the indicator that $w$ and $v$ are not adjacent. The implicit constant in the error term depends on $p$, and deteriorates as $p\rightarrow 0$. One way to interpret this result is that the hitting times concentrates around two values, and such a concentration can be used to recover the adjacency structure of the graphs. \\ Our main result, Theorem \ref{thm:main}, allows to derive similar formulas in a variety of situations. As a first example, we obtain the following corollary that extends \eqref{steiner} to the case where $p\rightarrow 0$. 
\begin{corollary}\label{thm:er}
Let $G=(V,E)$ be an Erd\H{o}s-Rényi random graph with parameters $n$ and $p=p_n$ where 
$\log^5 n\cdot n^{-1}\cdot p^{-8}\rightarrow 0.
$
Then, with high probability, we have that for every pair $w\neq v$ the hitting time $T_{w,v}$ satisfies 
\begin{equation}\label{sharpenstein}
\mathbb E[T_{w,v}]=\frac{2|E|}{\deg v}-1+\frac{X+Y}{X}1_{w\not\sim v}+o(1)
\end{equation}
where, conditional on $\deg v$, the random variables $X$ and $Y$ are independent and 
\begin{align*}
X\sim \emph{Bin} \left(\deg v(n-\deg v-1),p\right)\, \quad Y\sim \emph{Bin}\left((n-\deg v-1)(n-\deg v-2),p\right).
\end{align*} 
\end{corollary}
The restriction on $p$ is likely an artifex of the proof, though there is a real obstacle at $p=n^{-1/2}$, which is the regime at which the diameter of a typical graph changes from $2$ to $3$. In first approximation, we have 
$$
\frac{X+Y}{X}\approx \frac{n}{\deg v}\approx \frac{1}{p}
$$
thus recovering the result in \cite{ott} when $p$ is fixed. In the regime of $p$ where Corollary \ref{thm:er}, this is also a refinement of the results in \cite{lowe, lowe2, lowe3}. In fact, they prove a law of large numbers and central limit theorem for $\mathbb E[T_{w,v}]$ in the case where $w$ is distributed according to the stationary distributed -- confirming a prediction from the physics literature \cite{sood}. Provided that $p$ decays according to \eqref{thm:er}, their result can be recovered from the leading term in \eqref{sharpenstein} using asymptotic of binomial coefficients.
\subsection{Stochastic block model}
In many situations, the Erd\H{o}s-Rényi model is unrealistic as it assumes that each pair of vertices interact with the same probability. A widely used variation, the stochastic block model \cite{abbe}, takes into account for situations where vertices can be partitioned into sets that we call communities. We write $w\simeq v$ to denote that $w$ and $v$ are in the same community. While what follows can be substantially generalized, we focus here on the case of two communities of size $n$ each, with each pair of vertices in the same community being connected with probability $p$, and each pair of vertices in different communities being connected with probability $q$. We call the resulting measure on graphs on $2n$ vertices $\mathcal G_{n,p,q}$. Notice that the case $p=q$ corresponds to the Erd\H{o}s-Rényi model with parameters $2n$ and $p$. An instance of $\mathbb E[T_{w,v}]$ for $G\sim G_{1000,0.8,0.2}$, for fixed $v$ and $w$ that ranges on $w\neq v$, is given in Figure \ref{fig1}.
\begin{center}\label{fig1}
\begin{figure}[h!]
{\includegraphics[width=0.8\textwidth]{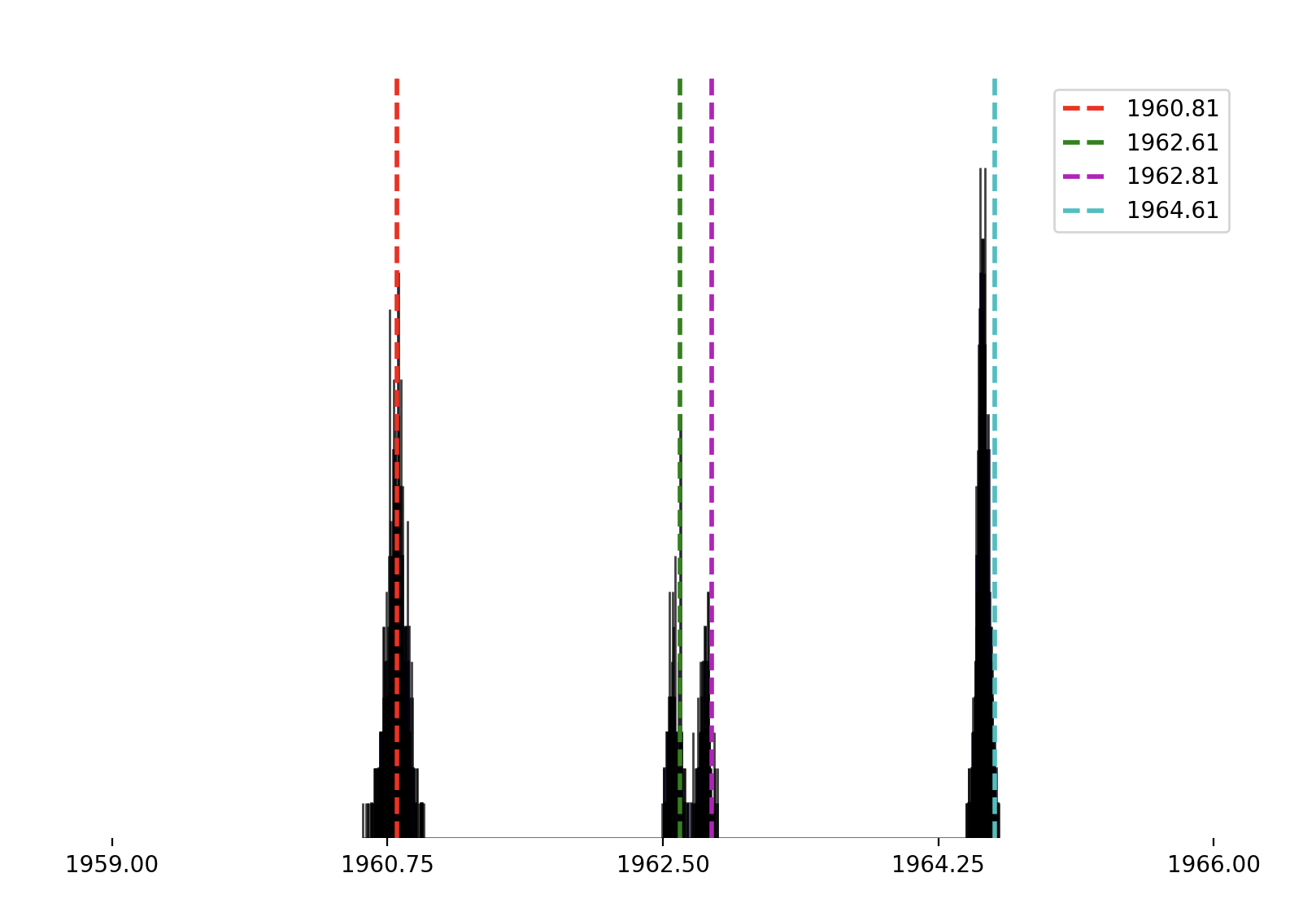}}
\caption{Average hitting times for an instance of the stochastic block model from different starting points. They cluster around four values that are predicted by our result (Theorem \ref{thm:sbm}). The clusters reflect the adjacency and community structure of the graph. 
}
\end{figure}
\end{center}

Recently, Löwe and Terveer \cite{lowe4} have shown a law of large numbers and a central limit theorem for $\mathbb E[T_{w,v}]$, when $w$ is stationary distributed. Their result is applicable in a wide regime -- allowing for multiple communities and $p, q$ that vary with $n$ -- but fails to capture the phenomenon in Figure \eqref{fig1}. An explanation if instead given by a corollary of our main Theorem \ref{thm:main}.
\begin{corollary}\label{thm:sbm}
Let $(V,E)=G\sim \mathcal G_{n,p,q}$ be a graph sampled according to the stochastic block model with $0<p,q<1$ fixed. Then, with high probability, we have that for every pair $w\neq v$ the hitting time $T_{w,v}$ satisfies
$$
\mathbb E[T_{w,v}]=\frac{2|E|}{\deg v}-1+\frac{2}{p+q}1_{w\not\sim v}
+
 \begin{cases}
-\frac{(q-p)^2}{(p+q)^2} \quad &\mbox{if}~w\simeq v \\
+\frac{p(q-p)^2}{q(p+q)^2} \quad &\mbox{if}~w\not\simeq v 
\end{cases} \quad
+\mathcal O_{p,q}\left(\sqrt{\frac{\log^5 n}{n}}\,\right)
$$
\end{corollary}
While some of the logarthmic factors may be a byproduct of the proof, we believe that the $\sqrt n$ behavior in the error term cannot be removed. One could also allow to vary $p, q$ with $n$, along the lines of Corollary \ref{thm:er}, though we will not pursue this direction. The case of more communities can also be handled using our main result. An illustration of the case with three communities is given in \eqref{fig2}. 
\begin{center}\label{fig2}
\begin{figure}[h!]
{\includegraphics[width=0.6\textwidth]{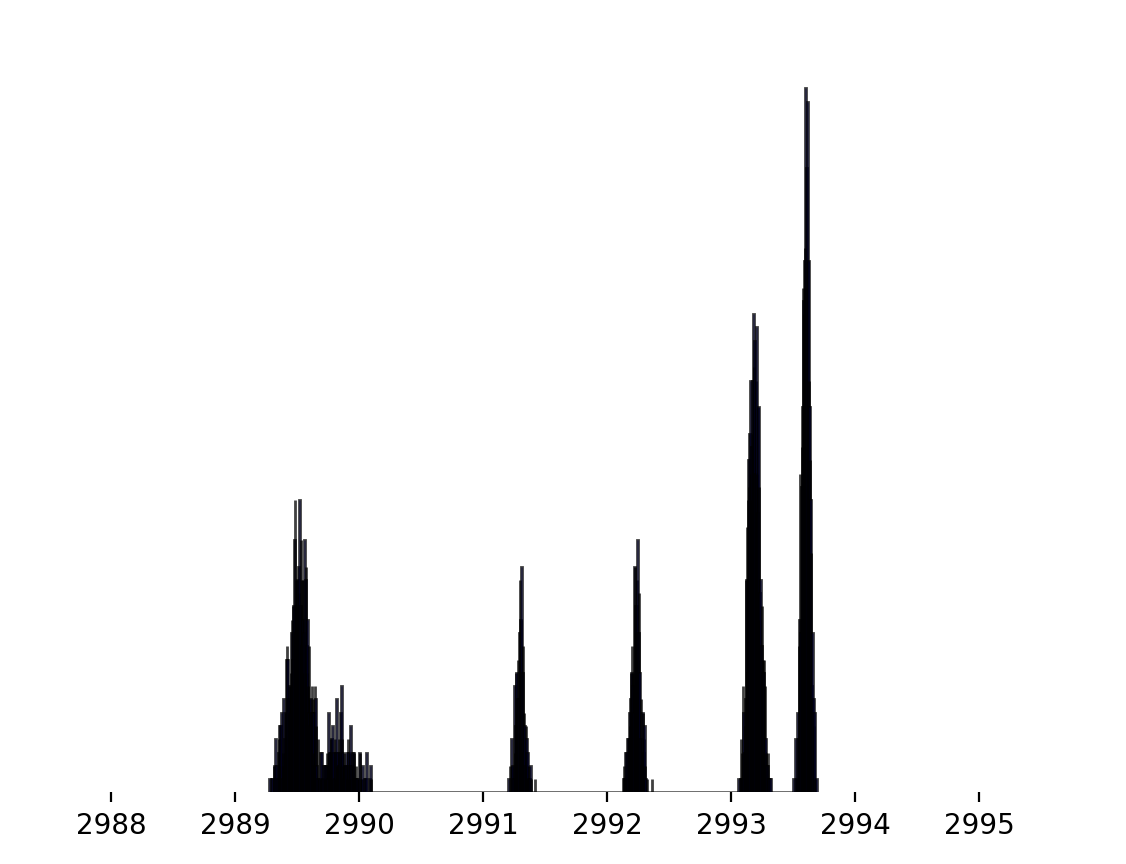}}
\caption{The average hitting times to a given vertex $v$ for an instance of the stochastic block model on three communities of a $1000$ individuals each. The hitting times cluster around six values. 
}
\end{figure}
\end{center}
As discussed in \cite{ott}, hitting times are closely related to other relevant quantities for the graph $G$, including spectral data \cite{persi, lov}, effective resistances \cite{kir}, and the uniform measure on the set of spanning tree\cite{lyo}. There is also a neat connection with the problem of community detection, that has attracted a lot of interest in the last two decades with a large body of literature in mathematics, physics, biology, computer science, see \cite{abbe, berg, sourav, bho, clauset, Ni, sax} and references therein. Classical approaches require an understanding of the spectrum of the adjacency matrix \cite{sax}, or are based on curvature methods \cite{Ni}. Our Corollary \ref{thm:sbm} illustrate a simple way to perform community detection in a stochastic block model by means of observing random messages travelling among individuals. 

\section{Main result}
\subsection{Set-up}
We consider a connected weighted graph $G$ with vertex set $V$ and non-negative edge weights $c_{w,w'}$ for each pair $1\leq w,w'\leq n$. We assume that the vertex set $V$ can be partitioned into $m+1$ disjoint parts $$\mathcal V=\{\{v\}=V_0, V_1, \ldots, V_m\},$$ 
which we use to define a new weighted graph with vertex set $\mathcal V$.
Given the symmetric weights $c_{w,w'}$ on the original graph, we define for $1\leq i,j\leq m$
$$
C_{ij}:=\sum_{w\in V_i, w'\in V_j}c_{w,w'}.
$$
This gives a weighted graph structure $\mathcal G$ on $\mathcal V$. We define 
$$
c_w:=\sum_{w'\in V}c_{w,w'},\quad C_{i}:=\sum_{0\leq j\leq m}C_{i,j}=\sum_{w'\in V_i}c_{w'}
$$
for the total weights at each vertex. We then introduce the random walks $X_k$ on $G$ and $Y_k$ on $\mathcal G$ with transition probabilities, respectively,
$$
p_{w,w'}:=\frac{c_{w,w'}}{c_w}, \quad q_{i,j}:=\frac{C_{i,j}}{C_i}.
$$
In particular, the $q_{i,j}s$ are a convex combination of the $p_{w, w'}$ and 
\begin{equation}\label{convextrick}
\min_{w\in V_i, w'\in V_j} p_{w,w'} \leq q_{i,j}\leq \max_{w\in V_i,w'\in V_j} p_{w,w'}
\end{equation}
We write $\mu(k,w)$ for the probability distribution of the random walk $X_k$ after $k$ steps, starting from $w$. Similarly, we denote by $\nu(k,i)$ the probability distribution of the random walk $Y_k$ after $k$ steps, starting from $V_i$. We will abuse notation and write $Y_k=i$ for $Y_k=V_i$. Owing to the connectedness assumption, the stationary distributions are unique and given by $\mu$ and $\nu$, where
\begin{equation}\label{def:stationary}
\mu_w:=\frac{c_w}{\sum_{w'\in V} c_{w'}}, \quad \nu_i:=\frac{C_i}{\sum_{0\leq j\leq m} C_j}
\end{equation}
We now define the hitting times
$$
T_{w,v}=\{\min k\geq 0: X_k=v, X_0=w\}, \quad T'_{i,0}=\{\min k\geq 0: Y_k=0, Y_0=i\},
$$
Our goal is to show that we can approximate $T_{w,v}$ with $T_{i,0}$ whenever $w\in V_i$. To do so, we will make the following assumption.
\begin{assumption}\label{ass}
The graph $G$ satisfies the following bounds.
\begin{enumerate}
\item
There exists $\varepsilon\in [0,1)$ such that for all $0\leq i, j\leq m$ and all $w\in V_i$,
\begin{equation}\label{ass1}
\left|\sum_{w'\in V_j}p_{w,w'}-q_{i,j}\right|\leq \varepsilon \min_{w''\in V_i}\sum_{w'\in V_j}p_{w'',w'}.
\end{equation}
\item 
There exists $K_1, K_2>0$ such that 
\begin{equation}\label{ass2}
\max_{w'\in V, 0\leq j\leq m}\left(\mathbb E[T_{w',v}], \mathbb E[T'_{j,0}] \right)\leq K_1|V|, \quad \max_{w'\sim v}p_{w,v}\leq \frac{K_2}{|V|}
\end{equation}
\item 
If $d_{\emph{TV}}$ denote the total variation distance between two probability measures, then we have for some $\lambda\in [0,1)$
\begin{equation}\label{ass3}
\max_{w\in V} d_{\emph{TV}}(\mu(1,w), \mu)\leq \lambda.
\end{equation}
\end{enumerate}
\end{assumption}
In many situations, that include the stochastic block model $\mathcal G_{n,p,q}$ with general parameters $0<p,q<1$ fixed, we can choose $\mathcal V$ to be small, the parameters in the assumptions are easy to estimate, and to do so does not require any knowledge about spectral data. Some generic conditions are shown in Lemma \eqref{mariangela}. In fact, one can also allow $p$ and/or $q$, to vary with $n$, as we show in the proof of Corollary \eqref{thm:er}. We now proceed with a brief discussion of the meaning of the parameters.\\ 
The first assumption allows to conclude that
\begin{equation}\label{convenient}
(1-\varepsilon)\mathbb P(X_k\in V_j|X_{k-1}=w)\leq \mathbb P(Y_k=j|Y_{k-1}=i)
\end{equation}
for all $i, j\in\{0,\ldots, m\}$ and $w\in V_i$. In particular, if we define $\Gamma$ via
$$
\Gamma:V\rightarrow \{0,1\ldots, m\}, \qquad \Gamma(w)=i\quad  \Leftrightarrow \quad w\in V_i,
$$
we deduce that we can couple $\Gamma(X_k)$ and $Y_k$ so that if $\Gamma(X_{k-1})=Y_{k-1}$, then they will remain equal with probability at least $1-\varepsilon$.\\ 
The second assumption is an a priori bound on the hitting of vertex $v$ and on the chance of hitting $v$ from its neighbor. Notice that if $K_1$ and $K_2$ are of the same order, it implies that in first approximation the average hitting times have the same size, regardless of the starting position. This is the case for the stochastic block model with arbitrary number of communities, as long as the connectivity parameters are bounded away from zero.\\ 
The third assumption is a mixing condition, that guarantees that $X_k$ can be coupled with a stationary walker with probability at least $\lambda$ at each step. It is referred to as a Dobrushin-type of condition in the statistical physics community \cite{ruelle}. It is also closely related to the optimal transport-based notion of curvature for the Markov chain $X_k$ known as Ollivier-Ricci curvature \cite{oll}.\\
\subsection{Main result}
We are now ready to state our main result. In a nutshell, it shows that $T_{w,v}$ and $T'_{i,0}$ behave similarly up to an error term governed by Assumption \ref{ass}. The bound is explicit, and allows for neat formulas for $T_{w,v}$ for large structured graphs. 
\begin{theorem}\label{thm:main}
Let $G$ be a connected weighted graph satisfying Assumption \ref{ass}. Then, for all $1\leq i\leq m$ and $w\in V_i$, we have
$$
\left|\mathbb E[T_{w,v}]-\mathbb E[T'_{i,0}]\right|\leq \frac{16K_1K_2\varepsilon}{\log^2\lambda}\log^2\left(\varepsilon \cdot \frac{K_2}{|V|}\cdot (1-\sqrt{\lambda})\right)
$$
\end{theorem}
In fact, the result is stronger, as the proof provides a coupling between the hitting times that is small in the $L^1$ sense, thus providing convergence in distribution whenever the right side vanishes. In many applications, $K_1, K_2$ and $\lambda$ are positive constants, while $\varepsilon\rightarrow 0$ sufficiently fast. In such cases, the expectations of $T_{w,v}$ and $T'_{i,0}$ coincide up to $o(1)$ factors. The latter is typically much easier to estimate, as we will see in the proofs of Corollaries \ref{thm:er} and \ref{thm:sbm}.
\section{Main ingredients}
\subsection{A useful lemma}
In order to apply Theorem \eqref{thm:main} in concrete settings, we need good and practical bounds for $\lambda$ and $K_1, K_2$ given by Assumption \eqref{ass}. The following lemma will suffice for our purposes. Versions of them are known in the literature, but we report them with proofs for the sake of completeness and to tailor them with an eye to our applications. 
\begin{lemma}\label{mariangela}
Consider a connected graph $G$ on $n$ vertices satisfying Assumption \eqref{ass}. Assume that there exists $0\leq \delta<1$ and $0<D<1$ such that 
    $$
    (1-\delta)D\leq \frac{\deg(w)}{n}\leq (1+\delta)D, \quad \frac{(1+\delta)^2}{(1-\delta)}D\leq 1.
    $$
    for all $w\in V$, and that for $w\neq v$ we can bound for some $\alpha>0$
    $$
   \frac{|w': w'\sim w, w\sim v|}{\deg w}\geq \alpha.
   $$
 Then, we have the bounds 
    $$
    K_1\leq \frac{2D}{\alpha}(1+\delta), \quad K_2\leq \frac{1}{D(1-\delta)}, \quad \lambda\leq 1-D+o(\delta).
    $$
\end{lemma}    
\begin{proof}
The statement regarding $K_2$ follows at once from the assumption on the degree. As for the bound on the hitting time, start from any $w'
\in V$. We have by assumption that there is a likelihood of at least $\alpha$ of reaching a neighbor of $v$ in one step, and a chance of at least $D^{-1}(1+\delta)^{-1}$ of hitting $v$ at the next step. In other words, 
$$
\mathbb P(T_{w', v}>2)\leq 1-\frac{\alpha}{D(1-\delta)}
$$
Using the Markov property, this gives
$$
\mathbb P(T_{w', v}>2k)\leq \left(1-\frac{\alpha}{nD(1+\delta)}\right)^k 
$$
and thus summing over $k$
$$
\mathbb E[T_{w', v}]\leq 2\frac{nD(1+\delta)}{\alpha}.
$$
as desired. Following the remark in \eqref{convextrick}, the same bounds hold for the graph $\mathcal G$ with the same quantities $D$, $\delta$ and $\alpha$.
\\ 
We move now to the bound on $\lambda$. Given the bounds on the degree, we can estimate
$$
(1-\delta)Dn^2\leq \sum_w \deg(w)\leq (1+\delta)Dn^2
$$
and thus we deduce that the stationary distribution $\mu$ satisfy 
$$
\frac{1}{n}\frac{1-\delta}{1+\delta}\leq \mu_w\leq \frac{1+\delta}{1-\delta}\frac{1}{n}.
$$
for all $w$. Given any other vertex $w'\sim w$, we have
$$
\mu_{w'}\leq \frac{1+\delta}{1-\delta}\frac{1}{n}\leq \frac{D(1+\delta)^2}{1-\delta}\frac{1}{\deg w}\leq \mu_{w'}(1,w)
$$
In particular, this implies that $\mu\leq \mu(1,w)$ for all $w'\sim w$. Therefore, by definition of total variation distance  
\begin{align*}
\lambda&=d_{\emph{TV}}(\mu, \mu(1,w'))=\sum_{w\not\sim w'}\mu_{w} \leq (n-\deg(w'))\max_w \mu_w \\&\leq n\left(1-D(1-\delta)\right)\frac{1}{n}\frac{1+\delta}{1-\delta}=1-D+o(\delta). 
\end{align*}
\end{proof}
\subsection{Relationship between the stationary distributions}
We now show a connection between the stationary distributions $\mu$ and $\nu$ for $X_k$ and $Y_k$, introduced in \eqref{def:stationary}. In what follows, we recall that $\Gamma$ denotes the map sending $w\in V_{i}$ to $i$. In most cases, the transition rates of $\Gamma(X_k)$ depend on the specific location $X_k$ we are at, though our Assumption \ref{ass} suggests that the dependence is mild. For stationary starting configurations we can say even more.
\begin{lemma}\label{lem:1}
Let $X_0\sim \mu$ and $Y_0\sim \nu$. Then, the pair $(\Gamma(X_k), \Gamma(X_{k+1}))$ is equal in distribution to $(Y_k, Y_{k+1})$ for all $k$. Moreover, 
$$
\max_i d_{\emph{TV}}(\nu(1,i),\nu)\leq\max _w d_{\emph{TV}}(\mu(1,w), \mu)
$$
\end{lemma}
\begin{proof}
Since both chains are stationary, it suffices to show the statement for $k=0$. In this case, we notice that for every $i$ we have 
$$
\mathbb P(\Gamma(X_0)=i)=\sum_{w\in V_i}\mu_w=\sum_{w\in V_i}\frac{c_w}{\sum_{w'\in V}c_{w'}}=\frac{C_i}{\sum_{0\leq j\leq m} C_{j}}=\nu_i=\mathbb P(Y_0=i)
$$
To check the distribution of pairs are the same, it suffices to show that the transition probabilities coincide, i.e., we need to show
$$
\mathbb P(\Gamma(X_{k+1})=j|\Gamma(X_k)=i)=q_{i,j}.
$$
This is just an easy computation. 
$$
P(\Gamma(X_{k+1})=j|\Gamma(X_k)=i)=\sum_{w\in V_i, w'\in V_j}\frac{c_{w,w'}}{c_w}\frac{c_w}{\sum_{a\in V_i}c_a}=\frac{C_{i,j}}{C_i}=q_{i,j}.
$$
It remains to prove the total variation bound. Fix $1\leq i\leq m$ and consider $\tilde X_0$ that is distributed according to the measure $\tilde \mu$
$$
\tilde \mu(w)=\frac{\mu_w1_{w\in V_i}}{\mu(V_i)},
$$
i.e. 
Our previous computation shows that the pair $(\Gamma(X_0), \Gamma(X_1))$ has the same law of $(Y_0, Y_1)$, where $Y_0=i$. Therefore, we have
$$
d_{\emph{TV}}(\nu(1,i),\nu)\leq d_{\emph{TV}}(\tilde \mu, \mu)\leq \max_{w\in V}d_{\emph{TV}}(\mu(1,w), \mu)
$$
where we used the convexity of total variation distance in the last inequality.
\end{proof} 
The lemma does not extend to triples, in the following sense: if we condition on $X_0=i$ and $X_1=j$, the distribution of $X_2$ is typically not identical to that of $Y_2$ conditioned on $Y_0=i$ and $Y_1=j$. 
\subsection{Construction of the coupling}\label{coupling}
Start with $X_0=w\in V_i$ and $Y_0=i$, so that in particular $\Gamma(X_0)=Y_0$. Heuristically, we want to construct a coupling such that $\Gamma(X_k)$ and $Y_k$ agree for as long as possible. In case things go wrong, we want to re-couple them as quick as possible, and we will do so by guaranteeing that both the chains reach their stationary states as fast as they can.
\\ \\ 
Let $w=X_0, X_1, \ldots$ be a realization of the Markov chain $X_k$. To describe the coupling, we introduce the independent random variables 
$$
\tau_d\sim \text{Geo}(\varepsilon),\quad  \tau_c, \tau'_c\sim \text{Geo}(1-\lambda).
$$
First we notice that by definition of $\lambda$ and the coupling interpretation of total variation distance (see, e.g., Theorem $5.2$ in \cite{peres}), we can construct a stationary walker $\tilde X_k$ so that
$$
X_k=\tilde X_k, \quad k\geq \tau_c. 
$$
On the other hand, owing to \eqref{convenient} we can guarantee that
$$
Y_k=\Gamma(X_k), \quad k<\tau_d
$$
By means of Lemma \ref{lem:1}, if $\tau_c<\tau_d$ we can then set $Y_k=\Gamma(X_k)$ for all $k$. If, instead, $\tau_d\leq \tau_c$, we define $Y_k$ for $k>\tau_d$ as follows.  
\begin{enumerate}
\item For $\tau_d<k\leq \tau_c$, move the two chains independently. If only one of the events $\{T=k\}$ or $\{T'=k\}$ occurs in this range, then move the two chains independently after that. 
\item For $\tau_c<k\leq \tau_c+\tau'_c$, use Lemma \ref{lem:1} to couple $Y_k$ with $X_k=\tilde X_k$ so that $\Gamma(X_{\tau_c+\tau'_c})=Y_{\tau_c+\tau'_c}$. If only one of the events $\{T=k\}$ or $\{T'=k\}$ occurs in this range, then move the two chains independently after that. 
\item If $k>\tau_c+\tau_c'$ and neither $T$ nor $T'$ occurred by time $k$, set $Y_k=\Gamma(X_k)=\Gamma(\tilde X_k)$, which is possible owing to Lemma \ref{lem:1}.
\end{enumerate}

In particular, our coupling guarantees that the inclusion of events
\begin{equation}\label{inclusion}
\{T\neq T'\}\subset \{T\neq T', \tau_d\leq \min(T,T')<\tau_c+\tau'_c\}.
\end{equation}
Moreover, if we write  
\begin{equation}\label{indephelps}
T-T'=\left(\max(T,T')-\min(T,T')\right)1_{T\neq T'},
\end{equation}
we see that, because of the construction of the coupling and the Markov property, $\max(T,T')-\min(T,T')$ depends on $\min(T,T')$ only through $X_{\min(T,T')}, Y_{\min(T,T')}$. \\ \\

\section{Proofs}
\subsection{Proof of Theorem \ref{thm:main}}
\begin{proof}
We start by estimating $\mathbb P(T\neq T')$. Using the inclusion of events in \eqref{inclusion}, we can write
$$
\mathbb P(T\neq T')\leq \sum_{k=1}^{\infty} \mathbb P(T\neq T',  \tau_d\leq k<\tau_c+\tau'_c, \min(T,T')=k)
$$
For $k$ small (to be decided), we will estimate 
using a union bound
$$
\mathbb P(\min(T,T')=k, T\neq T', \tau_d\leq k)\leq k\varepsilon\max_{1\leq s\leq k}\mathbb P(\min(T,T')=k, T\neq T'|\tau_d=s)
$$
For each $1\leq s\leq k-1$, we can further condition on $X_{k-1}$ and $Y_{k-1}$ and use a union bound to deduce
$$
\mathbb P(\min(T,T')=k|\tau_d=s)\leq \max_{w'\sim v} p_{w', v}+\max_{j\sim 0}q_{j,0}\leq 2\frac{K_2}{|V|}
$$
where the last inequality follows from \eqref{convextrick} and the definition \eqref{ass2} of $K_2$.
If, instead, $\tau_d=k$, we can write 
\begin{align*}
\mathbb P(\min(T,T')=k, T\neq T'|\tau_d=k)&\leq \mathbb P(\Gamma(X_k)=0|\tau_d=k, \Gamma(X_{k-1})\neq 0)\\&+\mathbb P(Y_k=0|\tau_d=k, Y_{k-1}\neq 0)
\end{align*}
The first term can be handled using the fact that $\tau_d$ is independent of $X_k$, and thus
$$
\mathbb P(\Gamma(X_k)=0|\tau_d=k, \Gamma(X_{k-1})\neq 0)\leq \max_{w'\sim v}p_{w', w}\leq \frac{K_2}{|V|} 
$$
As for the second term, we can use \eqref{ass1} and \eqref{convenient} to write 
$$
\mathbb P(Y_k=0|\tau_d=k, Y_{k-1}\neq 0)\leq  \mathbb P(\Gamma(X_k)=0|\Gamma(X_{k-1})\neq 0)\leq \max_{w'\sim v}p_{w', w}\leq \frac{K_2}{|V|}. 
$$
Overall, we have deduced the bound 
$$
\mathbb P(\min(T,T')=k, T\neq T', \tau_d\leq k)\leq 2k\varepsilon \frac{K_2}{|V|}.
$$
that will be useful for small $k$. As for large $k$, we use the bound
\begin{align*}
\mathbb P(\min(T,T')=k, T\neq T', k<\tau_c+\tau_c')\leq \mathbb P(k<\tau_c+\tau_c')\leq 2\mathbb P(\tau_c>\frac{k}{2})\leq 2(\sqrt\lambda)^k
\end{align*}
Collecting all bounds, we obtain that for every positive integer $z$ we can bound
\begin{align*}
\mathbb P(T\neq T')&\leq \sum_{k=1}^{\infty} \mathbb P(\min(T,T')=k, T\neq T',\tau_d\leq k<\tau_c+\tau_c')\\&\leq \sum_{k=1}^{z-1}\mathbb P(\min(T,T')=k, T\neq T', \tau_d\leq k)+\sum_{k=z}^{\infty}\mathbb P(\tau_c+\tau'_c>k)
\\&\leq 2z^2\varepsilon \frac{K_2}{|V|} +2\frac{(\sqrt\lambda)^z}{1-\sqrt\lambda}
\end{align*}
The choice of 
$$z=\frac{2}{\log\lambda}\cdot \log\left(\varepsilon \frac{K_2}{|V|}(1-\sqrt{\lambda})\right)
$$
gives 
$$
\mathbb P(T\neq T')\leq \frac{16\varepsilon}{\log^2\lambda}\cdot \frac{K_2}{|V|}\cdot \log^2\left(\varepsilon \frac{K_2}{|V|}(1-\sqrt{\lambda})\right)
$$
Using the identity in \eqref{indephelps} and the definition of $K_1$ from \eqref{ass2}, we conclude
\begin{align*}
|\mathbb E[T]-\mathbb E[T']|&\leq \mathbb E|T-T'|\leq \max_{w'\in V, 0\leq j\leq m}\left(\mathbb E[T_{w',v}], \mathbb E[\tilde T_{j,0}] \right)\mathbb P(T\neq T')\\&\leq K_1|V|\mathbb P(T\neq T')\leq \frac{16K_1K_2\varepsilon}{\log^2\lambda}\log^2\left(\varepsilon \cdot \frac{K_2}{|V|}\cdot (1-\sqrt{\lambda})\right).
\end{align*} 
\end{proof}
\subsection{Proofs of Corollaries \ref{thm:sbm} and \ref{thm:er}}
In the proofs of both Corollaries, we will make heavy use of the classical Chernoff bound for binomial random variables. We have that for any $N$ and $0<r<1$ (possibly depending on $N$) a random variable $X\sim \text{Bin}(N, r)$ satisfies 
$$
\mathbb P\left(|X-Nr|\geq \sqrt{3cNr\log N}\right)\leq \frac{1}{N^c}.$$
for every $c$ (while this is not optimal for $r$ close to $1$, it will suffice for our purposes). In particular, if we have a family of random variables of size at most polynomially in $N$, by taking $c$ sufficiently large we can ensure that, with high probability, they all lie in a neighbor of $Nr$ of size at most $\mathcal O\left(\sqrt{Nr\log N}\right)$, uniformly in $r$.
\begin{proof}[Proof of Corollary \ref{thm:sbm}]
Given $G\sim G_{n,p,q}$ and a vertex $v$, let $A$ denote the set of vertices distinct from $v$ in the same community of $v$, and let $B$ denote the set of vertices that are not in the same community of $v$ (thus $|A|=n-1$, $|B|=n$). Notice that, for parameters $p, q$ that are bounded away from zero, we are guaranteed that the graph have diameter $2$ with high probability \cite{chu}. Now, define
$\mathcal V=\{V_0, V_1, V_2, V_3, V_4\}$ where 
\begin{align*}
&V_0=\{v\}, \\ & V_1=\{w\in A: w\sim v\},\\ & V_2=\{w\in B: w\sim v\} \\ &V_1=\{w\in A: w\not\sim v\},\\ & V_2=\{w\in B: w\not\sim v\}
\end{align*}
In particular, this gives
$$
|V_1|=\text{Bin}(n-1,p), \quad |V_2|=\text{Bin}(n,q), \quad |V_3|=n-1-|V_1|, \quad |V_4|=n-|V_2|.
$$
and obviously $|V_0|=1$. We now proceed to show that Theorem \ref{thm:main} is in force with $\lambda, K_1, K_2$ that are independent of $n$, and $\varepsilon\rightarrow 0$. In what follows, the implicit constant in $\mathcal O$ may depend on $p$ and $q$. First, notice that for every vertex $w$ 
$$
\deg(w)=\text{Bin}(n-1, p)+\text{Bin}(n,q), 
$$
and thus a Chernoff bound gives with high probability
\begin{equation}\label{degreebound}
\frac{\deg(w)}{2n}=\frac{p+q}{2}\left(1+\mathcal O\left(\sqrt\frac{\log n}{n}\,\right)\right)
\end{equation}
for all $w\in V$. A similar argument implies that with high probability $w\in V_1$ the fraction of edges leading to $V_2$, whose size is $$|V_2|=nq\left(1+\mathcal O\left(\sqrt{\frac{\log n}{n}}\,\right)\right)$$ 
satisfies the uniform bound
$$
\frac{|w':w'\in V_2, w'\sim w|}{\deg w}=\frac{\text{Bin}(|V_2|, q)}{\text{Bin}(n-1,p)+\text{Bin}(n,q)}=\frac{q^2}{p+q}\left(1+\mathcal O\left(\sqrt\frac{\log n}{n}\right)\right),
$$
which is independent of $w\in V_1$.
A similar argument for each pair $0\leq i, j\leq 4$ shows that Assumption \ref{ass} holds with 
$$
\varepsilon=\mathcal O\left(\sqrt\frac{\log n}{n}\right).
$$
We now proceed to bound $\lambda, K_1, K_2$, and we do so by applying Lemma \ref{mariangela}. In fact, we have 
from \eqref{degreebound} that 
$$D=\frac{p+q}{2},\quad  \delta=\mathcal O\left(\sqrt\frac{\log n}{n}\right).$$
Moreover, we also have
$$
\alpha=\frac{|w':w'\sim v, w
\sim w|}{\deg w}\geq \min(p,q)\left(1+\mathcal O\left(\sqrt\frac{\log n}{n}\right)\right)
$$
Applying Lemma \eqref{mariangela}, we deduce the bounds 
$$
\quad K_1\leq\frac{p+q}{pq}+o(1), \quad K_2\leq \frac{2}{p+q}+o(1),\quad   \lambda\leq 1-\frac{p+q}{2}+o(1). 
$$
Using Theorem \ref{thm:main}, we conclude  
$$
\mathbb E[T_{w,v}]-\mathbb E[T'_{i, 0}]=\mathcal O\left(\varepsilon\left(\log\left(\frac{\varepsilon}{n}\right)\right)^2\right)=\mathcal O\left(\sqrt\frac{\log^5 n}{n}\right)
$$
It remains to compute $T'_{i,0}$ for $1\leq i\leq 4$. These are just the hitting times for a weighted graph with explicit transition probabilities. 
$$
Q=\begin{pmatrix}
0 & \frac{p}{(p+q)} & \frac{q}{(p+q)} & 0 & 0 & \\
\frac{1}{(p+q)n} & \frac{p^2}{p+q}& \frac{q^2}{p+q}& \frac{p(1-p)}{p+q}& \frac{q(1-q)}{p+q}\\ 
\frac{1}{(p+q)n} & \frac{pq}{p+q}& \frac{pq}{p+q}& \frac{q(1-p)}{p+q}& \frac{p(1-q)}{p+q}\\ 
0 &\frac{p^2}{p+q}& \frac{q^2}{p+q}& \frac{p(1-p)}{p+q}& \frac{q(1-q)}{p+q}\\
0 & \frac{pq}{p+q}& \frac{pq}{p+q}& \frac{q(1-p)}{p+q}& \frac{p(1-q)}{p+q}\\
\end{pmatrix}+\mathcal O\left(\sqrt \frac{\log n }{n}\right)
$$ 
where the error term is uniform over all entries. In addition to that, we notice that the return time is exactly the inverse of the stationary distribution at $v$. Therefore, we have 
$$
\frac{2|E|}{\deg v}=1+q_{01}T'_{1,0}+q_{02} T'_{2,0}. 
$$
This suggests writing
$
T_{i,0}'=\frac{2|E|}{\deg v}-1+\tilde T_i
$
and, using the bound 
$$
\left(\frac{2|E|}{\deg v}-1\right)\frac{1}{(p+q)n}=\frac{2}{p+q}+\mathcal O\left(\sqrt \frac{\log n }{n}\right)
$$
we deduce 
$$
\begin{pmatrix}
\tilde T_1 \\ \tilde T_2 \\ \tilde T_3 \\ \tilde T_4
\end{pmatrix}=\begin{pmatrix}
1-\frac{2}{p+q}\\ 1-\frac{2}{p+q} \\ 1 \\ 1
\end{pmatrix}+B_{p,q}\begin{pmatrix}
\tilde T_1 \\ \tilde T_2 \\ \tilde T_3 \\ \tilde T_4
\end{pmatrix}+\mathcal O\left(\sqrt \frac{\log n}{n}\right)
$$
where the matrix $B_{p,q}$ is given by 
$$
B_{p,q}:=\begin{pmatrix}
\frac{p^2}{p+q}& \frac{q^2}{p+q}& \frac{p(1-p)}{p+q}& \frac{q(1-q)}{p+q}\\ 
 \frac{pq}{p+q}& \frac{pq}{p+q}& \frac{q(1-p)}{p+q}& \frac{p(1-q)}{p+q}\\ 
\frac{p^2}{p+q}& \frac{q^2}{p+q}& \frac{p(1-p)}{p+q}& \frac{q(1-q)}{p+q}\\ \frac{pq}{p+q}& \frac{pq}{p+q}& \frac{q(1-p)}{p+q}& \frac{p(1-q)}{p+q}\\
\end{pmatrix}
$$
The linear system can be solved explicitly, leading to 
$$
\begin{pmatrix}
\tilde T_1 \\ \tilde T_2 \\ \tilde T_3 \\ \tilde T_4
\end{pmatrix}=\begin{pmatrix}
-\frac{(p-q)^2}{(p+q)^2} \\ \frac{p(p-q)^2}{q(p+q)^2} \\ -\frac{(p-q)^2}{(p+q)^2}+\frac{2}{p+q} \\ \frac{q(p-q)^2}{p(p+q)^2}+\frac{2}{p+q}
\end{pmatrix}+\mathcal O\left(\sqrt \frac{\log n }{n}\right)
$$
\end{proof}
\begin{proof}[Proof of Corollary \ref{thm:er}]
First, we observe that in this regime for $n$ and $p$ the graph has diameter $2$ with high probability \cite{chu}. The proof now is similar to that of Theorem \ref{thm:sbm}, but we need to be more careful in keeping track of the various parameters. We have $$V_0=\{v\}, \quad V_1=\{w: w\sim v\}\quad V_2=V\setminus (V_1\cup V_0).$$
In the following, all statement follows from a Chernoff bound, and thus holds with high probability. For all vertices $w$, including $v$, we have
$$
\deg(w)=pn\left(1+\mathcal O\left(\sqrt\frac{\log n}{np}\right)\right).
$$
In particular, 
$$
|V_1|=\deg(v)=pn\left(1+\mathcal O\left(\sqrt\frac{\log n}{np}\right)\right).
$$
The same argument shows that for every $w\neq v$ we have
\begin{align*}
&\sum_{w'\in V_1}p_{w,w'}=\frac{\text{Bin}(|V_1|-1_{w\in V_1}, p)}{\deg(w)}=p\left(1+\mathcal O\left(\sqrt \frac{\log n}{np^2}\right)\right) \\&
\sum_{w'\in V_2}p_{w,w'}=\frac{\text{Bin}(|V_2|-1_{w\in V_2}, p)}{\deg(w)}=(1-p)\left(1+\mathcal O\left(\sqrt \frac{\log n}{np}\right)\right)
\end{align*}
and for $w\in V_1$
$$
p_{w,v}=\frac{1}{\deg w}=\frac{1}{np}\left(1+\mathcal O\left(\sqrt{\frac{\log n}{np}}\right)\right)
$$
In particular, we obtain that Assumption $\ref{ass1}$ holds with 
$$
\varepsilon=\mathcal O\left(\sqrt\frac{\log n}{np^2}\right)
$$
Similarly, we can apply Lemma \ref{mariangela} with 
$$D=p,\quad \alpha=p, \quad \delta=\mathcal O\left(\sqrt\frac{\log n}{np}\right), $$
Our assumption on $p$ guarantees that $\delta=o(p)$ and thus we deduce 
$$
K_1\leq 2+o(1), \quad K_2\leq \frac{1+o(1)}{p}, \quad \lambda=1-p+o(p)
$$
and thus $$
K_1K_2\varepsilon=\mathcal O\left(\sqrt\frac{\log n}{np^4}\right)
$$
On the other hand, we also have
$$
\frac{\log(\varepsilon \frac{K_2}{|V|}(1-\sqrt{\lambda}))}{
\log\lambda}=\mathcal O\left(\frac{\log n}{p}\right)
$$
Combining all the ingredients, Theorem \eqref{thm:main} is in force and gives a bound
$$
\mathbb E[T_{w,v}]-\mathbb E[T'_{i, 0}]=\mathcal O\left(\sqrt\frac{\log n}{np^4}\frac{(\log n)^2}{p^2}\right)=\mathcal O\left(\sqrt\frac{\log n^5}{np^8}\right)
$$
that goes to zero by our assumption. It remains to compute the hitting times for $T_{1,0}$ and $T_{2,0}$. From the return time to $0$, we can deduce
$$
\frac{2|E|}{\deg(v)}=1+T_{1,0},
$$
while from 
$$
T_{2,0}=1+q_{2,1}T_{1,0}+(1-q_{2,1})T_{2,0}
$$
we obtain that the difference of the hitting times is
$$
T_{2,0}-T_{1,0}=\frac{1}{q_{2,1}}.
$$
It remains to compute $q_{2,1}$. If we let $X$ denote the number of edges between $V_1$ and $V_2$, and by $Y$ denote the number of edges between $V_2$ and itself, we obtain $q_{2,1}=\frac{X}{X+Y}$. By the very definition of the model, $X$ and $Y$ are independent and  
\begin{align*}
X\sim \text{Bin} \left(\deg v(n-\deg v-1),p\right)\, \quad Y\sim \text{Bin}\left((n-\deg v-1)(n-\deg v-2),p\right).
\end{align*} 
\end{proof}
\section*{Acknowledgments}
We warmly thank Stefan Steinerberger for many useful discussions. We were also supported by an AMS-Simons travel grant.


\begin{thebibliography}{10}
\bibitem{abbe} E. Abbe. Community detection and stochastic block models: recent developments. Journal of Machine Learning Research 18.177 (2018): 1-86.

\bibitem{berg} C. T. Bergstrom, M. Rosvall.  Maps of random walks on complex networks reveal community structure. Proc. Natl. Acad. Sci. USA 105, 1118–1123 (2008).

\bibitem{bho} S. S. Bhowmick, B. S. Seah. Clustering and summarizing protein-protein interaction networks: A survey. IEEE Trans. Knowl. Data Eng. 28, 638–658 (2015).

\bibitem{sourav} S. Chatterjee, K. Rohe, B. Yu, Spectral clustering and the high-dimensional stochastic blockmodel. Ann. Statist. 39(4): 1878-1915 (2011).

\bibitem{chu} F. Chung, L. Linyuan, The diameter of sparse random graphs. Advances in Applied Mathematics 26.4 (2001): 257-279.

\bibitem{clauset} A. Clauset, M. E. Newman, C. Moore. Finding community structure in very large networks. Phys. Rev. E 70, 066111 (2004).


\bibitem{gestalt} A. Desolneux, L. Moisan, and J. Morel. From gestalt theory to image analysis: a probabilistic approach. Vol. 34. Springer Science \& Business Media, 2007.

\bibitem{persi} P. Diaconis, L. Miclo. On quantitative convergence to quasi-stationarity. Annales de la Faculte des sciences de Toulouse: Mathematiques. Vol. 24. No. 4. 2015.

\bibitem{Ni} J. Gao, CC. YY. Lin, F. Luo, CC. Ni, Community Detection on Networks with Ricci Flow, Sci Rep 9, 9984 (2019)

\bibitem{hel} A. Helali and M. L\"owe, 
Hitting times, commute times, and cover times for random walks on random hypergraphs.
Statist. Probab. Lett. 154 (2019), 108535, 6 pp.





\bibitem{kir} G. Kirchhoff,
\"Uber die Aufl\"osung der Gleichungen, auf welche man bei der Untersuchung der linearen Vertheilung
galvanischer Str\"ome gef\"uhrt wird. Ann. Phys. Chem. 72 (1847), 497–508

\bibitem{klee} V. Klee and D. Larman, Diameters of random graphs, Canadian Journal of Mathematics 33.3 (1981): 618-640.

\bibitem{peres} D. Levin, Y. Peres. Markov chains and mixing times. Vol. 107. American Mathematical Soc., 2017.

\bibitem{lov} L. Lov\'asz,
Random walks on graphs: a survey. (English summary) Combinatorics, Paul Erdős is eighty, Vol. 2 (Keszthely, 1993), 353--397,
Bolyai Soc. Math. Stud., 2, János Bolyai Math. Soc., Budapest, 1996.

\bibitem{lowe} M. L\"owe and F. Torres, On hitting times for a simple random walk on dense Erd\H{o}s-R\'enyi random graphs. Statistics \& Probability Letters 89 (2014), p. 81--88.

\bibitem{lowe2} M. L\"owe and S. Terveer, A Central Limit Theorem for the average target hitting time for a random walk on a random graph. arXiv preprint arXiv:2104.01053.

\bibitem{lowe3} M. L\"owe and S. Terveer, A central limit theorem for the mean starting hitting time for a random walk on a random graph. Journal of Theoretical Probability, 36(2), 779-810.

\bibitem{lowe4} M. L\"owe and S. Terveer, Spectral properties of the strongly assortative stochastic block model and their application to hitting times of random walks, arXiv preprint arXiv:2104.01053 


\bibitem{von} U. von Luxburg, A. Radl and M. Hein, Hitting and commute times in large random neighborhood graphs. The Journal of Machine Learning Research, 15 (2014), 1751--1798.

\bibitem{lyo} R. Lyons, Y. Peres, Probability on trees and networks.  Cambridge University Press, 2017.

\bibitem{oll} Y. Ollivier, Ricci curvature of metric spaces." Comptes Rendus Mathematique 345.11 (2007): 643-646.

\bibitem{ott} A. Ottolini, S. Steinerberger. Concentration of Hitting Times in Erd\H{o}s-Rényi graphs, to appear on Journal of Graph Theory. 

\bibitem{ruelle} D. Ruelle. Statistical mechanics: rigorous results. World Scientific, 1999.

\bibitem{sax} A. Saxena et al, A review of clustering techniques and developments, Neurocomputing 267 (2017): 664-681.

\bibitem{sood} V. Sood, S. Redner and D. ben-Avraham, 
First-passage properties of the Erdős-Renyi random graph. 
J. Phys. A 38 (2005), no. 1, 109--123.

\bibitem{sylvester} J. Sylvester, 
Random walk hitting times and effective resistance in sparsely connected Erd\H{o}s-R\'enyi random graphs. 
J. Graph Theory 96 (2021), no. 1, 44--84.




\end{thebibliography}
\end{document}